\documentclass[reqno,draft]{amsproc}
\usepackage{amssymb,tikz}
\usepackage{graphicx}
\usepackage{euscript}
\usepackage{mathrsfs} 
\usepackage{units}   
\usepackage{color}

\makeatletter
\@namedef{subjclassname@2010}{%
\textup{2010} Mathematics Subject Classification}
\makeatother

\numberwithin{equation}{section}

\newtheorem{thm}{Theorem}[section]
\newtheorem{cor}[thm]{Corollary}
\newtheorem{lem}[thm]{Lemma}

\newtheorem*{thm*}{Theorem}

\theoremstyle{remark}
\newtheorem{rem}[thm]{Remark}

\theoremstyle{definition}
\newtheorem{exa}[thm]{Example}
\newtheorem*{dfn*}{Definition}

\DeclareMathOperator{\D}{d}

\DeclareMathOperator{\dzii}{{\mathsf{Chi}}}

\DeclareMathOperator{\koo}{{\mathsf{root}}}

\DeclareMathOperator{\paa}{{\mathsf{par}}}

\newcommand*{\ascr}{\mathscr A}
\newcommand*{\borel}[1]{{\mathfrak B}(#1)}
\newcommand*{\cbb}{\mathbb C}

\newcommand*{\cfw}{{\mathsf C}}

\newcommand*{\dz}[1]{{\EuScript D}(#1)}
\newcommand*{\dzi}[1]{\dzii(#1)}

\newcommand*{\dzn}[1]{{\EuScript D}^\infty(#1)}

\newcommand*{\esf}{\mathsf{E}}

\newcommand*{\Ge}{\geqslant}
\newcommand*{\hh}{\mathcal H}
\newcommand*{\hsf}{{\mathsf h}}

\newcommand*{\kk}{\mathcal K}

\newcommand*{\lambdab}{{\boldsymbol\lambda}}

\newcommand*{\Le}{\leqslant}

\newcommand*{\nbb}{\mathbb N}
\newcommand*{\ogr}[1]{\boldsymbol B(#1)}

\newcommand*{\pa}[1]{\paa(#1)}

\newcommand*{\rbb}{\mathbb R}
\newcommand*{\rbop}{{\overline{\mathbb R}_+}}
\newcommand*{\slam}{S_{\boldsymbol \lambda}}
\newcommand*{\smalloplus}{\raise0pt\hbox{$\scriptscriptstyle \oplus$}}

\newcommand*{\tcal}{{\mathscr T}}

\newcommand*{\zbb}{\mathbb Z}

\begin{document}
   \title[Subnormal
operators with non-densely defined
powers]{Subnormal weighted shifts on directed
trees and composition operators in $L^2$ spaces
with non-densely defined powers}
   \author[P.\ Budzy\'{n}ski]{Piotr Budzy\'{n}ski}
   \address{Katedra Zastosowa\'{n} Matematyki,
Uniwersytet Rolniczy w Krakowie, ul.\ Balicka 253c,
PL-30198 Krak\'ow, Poland}
\email{piotr.budzynski@ur.krakow.pl}
   \author[P.\ Dymek]{Piotr Dymek}
   \address{Instytut Matematyki,
Uniwersytet Jagiello\'nski, ul.\ \L ojasiewicza 6,
PL-30348 Kra\-k\'ow, Poland}
\email{Piotr.Dymek@im.uj.edu.pl}
   \author[Z.\ J.\ Jab{\l}o\'nski]{Zenon Jan
Jab{\l}o\'nski}
   \address{Instytut Matematyki,
Uniwersytet Jagiello\'nski, ul.\ \L ojasiewicza 6,
PL-30348 Kra\-k\'ow, Poland}
\email{Zenon.Jablonski@im.uj.edu.pl}
   \author[J.\ Stochel]{Jan Stochel}
\address{Instytut Matematyki, Uniwersytet
Jagiello\'nski, ul.\ \L ojasiewicza 6, PL-30348
Kra\-k\'ow, Poland} \email{Jan.Stochel@im.uj.edu.pl}
   \thanks{The research of the first author was
supported by the NCN (National Science Center)
grant DEC-2011/01/D/ST1/05805. The research of
the third and the fourth author was supported by
the MNiSzW (Ministry of Science and Higher
Education) grant NN201 546438 (2010-2013).}
   \subjclass[2010]{Primary 47B20, 47B37;
Secondary 47B33} \keywords{(Weighted) composition operator in an $L^2$
space, weighted shift on a
directed tree, subnormal operator, non-densely defined power}
   \begin{abstract}
It is shown that for every positive integer $n$
there exists a subnormal weighted shift on a
directed tree (with or without root) whose $n$th
power is densely defined while its $(n+1)$th
power is not. As a consequence, for every
positive integer $n$ there exists a non-symmetric
subnormal composition operator $C$ in an $L^2$
space over a $\sigma$-finite measure space such
that $C^n$ is densely defined and $C^{n+1}$ is
not.
   \end{abstract}
   \maketitle
   \section{Introduction}
The question of when powers of a closed densely
defined linear operator are densely defined has
attracted considerable attention. In 1940 Naimark
gave a surprising example of a closed symmetric
operator whose square has trivial domain (see
\cite{nai}; see also \cite{cher} for a different
construction). More than four decades later,
Schm\"{u}dgen discovered another pathological
behaviour of domains of powers of symmetric
operators (cf.\ \cite{Schm}). It is well-known
that symmetric operators are subnormal (cf.\
\cite[Theorem 1 in Appendix I.2]{a-g}). Hence,
closed subnormal operators may have non-densely
defined powers. In turn, quasinormal operators,
which are subnormal as well (see \cite{Br} and
\cite{StSz1}), have all powers densely defined
(cf.\ \cite{StSz1}). In the present paper we
discuss the above question in the context of
subnormal weight\-ed shifts on directed trees and
subnormal composition operators in $L^2$ spaces
(over $\sigma$-finite measure spaces).

As recently shown (cf.\ \cite[Proposition
3.1]{j-j-s3}), formally normal, and consequently
symmetric, weighted shifts on directed trees are
automatically bounded and normal (in general,
formally normal operators are not subnormal, cf.\
\cite{b-c}). The same applies to symmetric
composition operators in $L^2$ spaces (cf.\
\cite[Proposition B.1]{b-j-j-sS}). Formally
normal composition operators in $L^2$ spaces,
which may be unbounded (see \cite[Appendix
C]{b-j-j-sS}), are still normal (cf.\
\cite[Theorem 9.4]{b-j-j-sC}). As a consequence,
all powers of such operators are densely defined
(see e.g., \cite[Corollary 5.28]{Schm2}).

The above discussion suggests the question of
whether for every positive integer $n$ there
exists a subnormal weighted shift on a directed
tree whose $n$th power is densely defined while
its $(n+1)$th power is not. A similar question
can be asked for composition operators in $L^2$
spaces. To answer both of them, we proceed as
follows. First, by applying a recently
established criterion for subnormality of
weighted composition operators in $L^2$ spaces
which makes no appeal to density of
$C^\infty$-vectors (see Theorem \ref{general}),
we show that a densely defined weighted shift on
a directed tree which admits a consistent system
of probability measures\footnote{\;i.e., a system
$\{\mu_v\}_{v \in V}$ of Borel probability
measures on $\rbb_+$ which satisfies
\eqref{consist6}.} is subnormal, and, what is
more, its $n$th power is densely defined if and
only if all moments of these measures up to
degree $n$ are finite (cf.\ Theorem \ref{wsi}).
The particular case of directed trees with one
branching vertex is examined in Theorem
\ref{glowne} and Corollary \ref{wndd}. Using
these two results, we answer both questions in
the affirmative (see Example \ref{glownyprz} and
Remark \ref{lebaa}). It is worth pointing out
that though directed trees with one branching
vertex have simple structure, they provide many
examples which are important in operator theory
(see e.g., \cite{j-j-s,j-j-s2}).

   Now we introduce some notation and
terminology. In what follows, $\zbb$, $\zbb_+$,
$\nbb$, $\rbb_+$ and $\cbb$ stand for the sets of
integers, nonnegative integers, positive
integers, nonnegative real numbers and complex
numbers, respectively. Set $\rbop=\rbb_+ \cup
\{\infty\}$. We write $\borel{\rbb_+}$ for the
$\sigma$-algebra of all Borel subsets of
$\rbb_+$. Given $t\in \rbb_+$, we denote by
$\delta_t$ the Borel probability measure on
$\rbb_+$ concentrated on $\{t\}$.

The domain of an operator $A$ in a complex
Hilbert space $\hh$ is denoted by $\dz{A}$ (all
operators considered in this paper are linear).
Set $\dzn{A} = \bigcap_{n=0}^\infty\dz{A^n}$.
Recall that a closed densely defined operator $A$
in $\hh$ is said to be {\em normal} if $AA^* =
A^*A$ (see \cite{b-s,Schm2,Weid} for more on this
class of operators). We say that a densely
defined operator $A$ in $\hh$ is {\em subnormal}
if there exists a complex Hilbert space $\kk$ and
a normal operator $N$ in $\kk$ such that $\hh
\subseteq \kk$ (isometric embedding) and $Ah =
Nh$ for all $h \in \dz{S}$. We refer the reader
to \cite{Con} and \cite{StSz3,StSz1,StSz4,StSz2}
for the foundations of the theory of bounded and
unbounded subnormal operators, respectively.
   \section{Weighted composition operators}
Assume that $(X,\ascr,\nu)$ is a $\sigma$-finite
measure space, $w\colon X \to \cbb$ is an
$\ascr$-measurable function and $\phi\colon X \to
X$ is an $\ascr$-measurable mapping. Define the
$\sigma$-finite measure $\nu_w\colon \ascr \to
\rbop$ by $\nu_w(\varDelta) = \int_{\varDelta}
|w|^2 \D\nu$ for $\varDelta \in \ascr$. Let
$\nu_w \circ \phi^{-1}\colon \ascr \to \rbop$ be the
measure given by $\nu_w \circ
\phi^{-1}(\varDelta)=\nu_w(\phi^{-1}(\varDelta))$
for $\varDelta \in \ascr$. Assume that $\nu_w
\circ \phi^{-1}$ is absolutely continuous with
respect to $\nu$. By the Radon-Nikodym theorem
(cf.\ \cite[Theorem 2.2.1]{Ash}), there exists a
unique (up to a.e.\ $[\nu]$ equivalence)
$\ascr$-measurable function $\hsf={\mathsf
h}_{\phi,w}\colon X \to \rbop$ such that
   \begin{align*}
\nu_w \circ \phi^{-1}(\varDelta) =
\int_{\varDelta} \hsf \D \nu, \quad \varDelta \in
\ascr.
   \end{align*}
Then the operator $\cfw = {\mathsf C}_{\phi,w}$
in $L^2(\nu)$ given by
   \begin{align}    \label{wco-def}
   \begin{aligned}
   \dz{\cfw} & = \{f \in L^2(\nu) \colon w \cdot
(f\circ \phi) \in L^2(\nu)\},
   \\
\cfw f & = w \cdot (f\circ \phi), \quad f \in
\dz{\cfw},
   \end{aligned}
   \end{align}
is well-defined (cf.\ \cite[Proposition
6]{b-j-j-sW}). Call $\cfw$ a {\em weighted
composition operator}. By \cite[Proposition
9]{b-j-j-sW}, $\cfw$ is densely defined if and
only if $\hsf < \infty$ a.e.\ $[\nu]$; moreover,
if this is the case, then
$\nu_w|_{\phi^{-1}(\ascr)}$ is $\sigma$-finite
and, by the Radon-Nikodym theorem, for every
$\ascr$-measurable function $f \colon X \to
\rbop$ there exists a unique (up to a.e.\
$[\nu_w]$ equivalence)
$\phi^{-1}(\ascr)$-measurable function $\esf(f) =
\mathsf{E}_{\phi,w}(f)\colon X \to \rbop$ such
that
   \begin{align*}
\int_{\phi^{-1}(\varDelta)} f \D \nu_w =
\int_{\phi^{-1}(\varDelta)} \esf(f) \D \nu_w,
\quad \varDelta \in \ascr.
   \end{align*}
We call $\esf(f)$ the {\em conditional
expectation} of $f$ with respect to
$\phi^{-1}(\ascr)$ (see \cite{b-j-j-sW} for more
information). A mapping $P\colon X \times
\borel{\rbb_+} \to [0,1]$ is called an {\em
$\ascr$-measurable family of probability
measures} if the set-function $P(x,\cdot)$ is a
probability measure for every $x \in X$ and the
function $P(\cdot,\sigma)$ is $\ascr$-measurable
for every $\sigma \in \borel{\rbb_+}$.

The following criterion (read:\ a sufficient
condition) for subnormality of unbounded weighted
composition operators is extracted from
\cite[Theorem 27]{b-j-j-sW}.
   \begin{thm} \label{general}
If $\cfw$ is densely defined, $\hsf > 0$ a.e.\
$[\nu_w]$ and there exists an $\ascr$-measurable
family of probability measures $P\colon X \times
\borel{\rbb_+} \to [0,1]$ such that
   \begin{align} \tag{CC} \label{cc}
\esf(P(\cdot, \sigma)) (x) = \frac{\int_{\sigma}
t P(\phi(x),\D t)}{\hsf(\phi(x))} \; \text{ for
$\nu_w$-a.e.\ $x \in X$}, \quad \sigma \in
\borel{\rbb_+},
   \end{align}
then $\cfw$ is subnormal.
   \end{thm}
   Regarding Theorem \ref{general}, recall that
if $\cfw$ is subnormal, then $\hsf
> 0$ a.e.\ $[\nu_w]$ (cf.\ \cite[Corollary 12]{b-j-j-sW}).
   \section{Weighted shifts on directed trees}
Let $\tcal=(V,E)$ be a directed tree ($V$ and $E$
stand for the sets of vertices and edges of
$\tcal$, respectively). Set $\dzi u = \{v\in
V\colon (u,v)\in E\}$ for $u \in V$. Denote by
$\paa$ the partial function from $V$ to $V$ which
assigns to each vertex $u\in V$ its parent
$\pa{u}$ (i.e.\ a unique $v \in V$ such that
$(v,u)\in E$). A vertex $u \in V$ is called a
{\em root} of $\tcal$ if $u$ has no parent. A
root is unique (provided it exists); we denote it
by $\koo$. Set $V^\circ=V \setminus \{\koo\}$ if
$\tcal$ has a root and $V^\circ=V$ otherwise. We say that $u \in V$ is a {\em branching vertex} of $V$, and write $u \in V_{\prec}$, if $\dzi{u}$ consists of at least two vertices.  We refer the reader to
\cite{j-j-s} for all facts about directed trees
needed in this paper.

By a {\em weighted shift on} $\tcal$ with weights
$\lambdab=\{\lambda_v\}_{v \in V^{\circ}} \subseteq
\cbb$ we mean the operator $\slam$ in $\ell^2(V)$
defined by
   \begin{align*} \begin{aligned}
\dz {\slam} & = \{f \in \ell^2(V) \colon
\varLambda_\tcal f \in \ell^2(V)\},
   \\
\slam f & = \varLambda_\tcal f, \quad f \in
\dz{\slam},
\end{aligned}
\end{align*}
where $\varLambda_\tcal$ is the mapping defined
on functions $f\colon V \to \cbb$ via
   \begin{align*}
(\varLambda_\tcal f) (v) =
   \begin{cases}
\lambda_v \cdot f\big(\pa v\big) & \text{ if }
v\in V^\circ,
   \\
0 & \text{ if } v=\koo.
   \end{cases}
   \end{align*}
(As usual, $\ell^2(V)$ is the Hilbert space of
square summable complex functions on $V$ with
standard inner product.) For $u \in V$, we define
$e_u \in \ell^2(V)$ to be the characteristic
function of the one-point set $\{u\}$. Then
$\{e_u\}_{u\in V}$ is an orthonormal basis of
$\ell^2(V)$.

   The following useful lemma is an extension of
part (iv) of \cite[Theorem 3.2.2]{j-j-s2}.
   \begin{lem} \label{ddn}
Let $\slam$ be a weighted shift on a directed tree
$\tcal=(V,E)$ with weights $\lambdab=\{\lambda_v\}_{v
\in V^{\circ}}$ and let $n\in \zbb_+$. Then $\slam^n$
is densely defined if and only if $e_u\in
\dz{\slam^n}$ for every $u\in V_{\prec}$.
   \end{lem}
   \begin{proof}
In view of \cite[Theorem 3.2.2(iv)]{j-j-s2},
$\slam^n$ is densely defined if and only if $e_u
\in \dz{\slam^n}$ for every $u\in V$. Note that
if $u \in V$ and $\dzi{u} = \{v\}$, then $e_u \in
\dz{\slam}$ and $\slam e_u = \lambda_v e_v$,
which implies that $e_u \in \dz{\slam^{n+1}}$
whenever $e_v \in \dz{\slam^n}$. In turn, if
$\dzi{u}=\varnothing$, then clearly $e_u \in
\dzn{\slam}$. Using the above and an induction
argument (related to paths in $\tcal$), we deduce
that $\slam^n$ is densely defined if and only if
$e_u \in \dz{\slam^n}$ for every $u\in
V_{\prec}$.
   \end{proof}
It is worth mentioning that if $V_{\prec} =
\varnothing$, then, by Lemma \ref{ddn} and
\cite[Theorem 3.2.2(iv)]{j-j-s2} (or by the proof
of Lemma \ref{ddn}), $\dzn{\slam}$ is dense in
$\ell^2(V)$. In particular, this covers the case
of classical weighted shifts and their adjoints.

Now we give a criterion for subnormality of
weighted shifts on directed trees. As opposed to
\cite[Theorem 5.1.1]{b-j-j-sA}, we do not assume
the density of $C^\infty$-vectors in the
underlying $\ell^2$-space. Moreover, we do not
assume that the underlying directed tree is
rootless and leafless, which is required in
\cite[Theorem 47]{b-j-j-sS}, and that weights are
nonzero. The only restriction we impose is that
the directed tree is countably infinite. This is
always satisfied if the weighted shift in
question is densely defined and has nonzero
weights (cf.\ \cite[Proposition 3.1.10]{j-j-s}).
   \begin{thm} \label{wsi}
Let $\slam$ be a weighted shift on a
countably infinite directed tree $\tcal=(V,E)$ with
weights $\lambdab=\{\lambda_v\}_{v \in V^{\circ}}$.
Suppose there exist a system $\{\mu_v\}_{v \in V}$ of
Borel probability measures on $\rbb_+$ and a system
$\{\varepsilon_v\}_{v\in V}$ of nonnegative real
numbers such that\/\footnote{\;We adopt the conventions
that $0\cdot \infty = \infty \cdot 0 = 0$,
$\frac{1}{0} = \infty$ and $\sum_{v\in \varnothing}
\xi_v = 0$.}
   \begin{align} \label{consist6}
\mu_u(\sigma) = \sum_{v \in \dzi{u}}
|\lambda_v|^2 \int_\sigma \frac{1}{t} \mu_v(\D t)
+ \varepsilon_u \delta_0(\sigma), \quad \sigma
\in \borel{\rbb_+}, \, u \in V.
   \end{align}
Then the following two assertions hold\/{\em :}
\begin{enumerate}
\item[(i)] if $\slam$ is densely defined, then
$\slam$ is subnormal,
\item[(ii)] if $n\in \nbb$, then $\slam^n$
is densely defined if and only
if\/\footnote{\;Here, and later,
$\int_0^{\infty}$ means integration over the set
$\rbb_+$.} $\int_0^\infty s^n \, \D \mu_u(s) <
\infty$ for all $u \in V_{\prec}$.
\end{enumerate}
   \end{thm}
   \begin{proof}
(i) Assume that $\slam$ is densely defined. Set $X=V$ and $\ascr = 2^V$. Let $\nu\colon \ascr
\to \rbop$ be the counting measure on $X$ ($\nu$
is $\sigma$-finite because $V$ is countable).
Define the weight function $w\colon X \to \cbb$
and the mapping  $\phi\colon X\to X$ by
   \begin{align*}
w(x) =
   \begin{cases}
\lambda_x & \text{ if } x\in V^\circ
   \\
0 & \text{ if } x=\koo
   \end{cases}
\quad \text{and} \quad \phi (x) =
   \begin{cases}
\pa x& \text{ if } x\in V^\circ
   \\
\koo & \text{ if } x=\koo.
   \end{cases}
   \end{align*}
Clearly, the measure $\nu_w \circ \phi^{-1}$ is
absolutely continuous with respect to $\nu$ and
\begin{align} \label{leba2}
\hsf(x)=\nu_w(\phi^{-1}(\{x\})) =\nu_w(\dzi{x}) = \sum_{y\in \dzi{x}}
|\lambda_y|^2, \quad x\in X.
\end{align}
Thus, by
\cite[Proposition 3.1.3]{j-j-s}, $\hsf(x) <
\infty$ for every $x\in X$. We claim that
$\hsf > 0$ a.e.\ $[\nu_w]$. This is the same as
to show that if $x\in V^\circ$ and
$\nu_w(\dzi{x}) = 0$, then $\lambda_x = 0$. Thus,
if $x\in V^\circ$ and $\nu_w(\dzi{x}) = 0$, then applying
\eqref{consist6} to $u=x$, we deduce that $\mu_x
= \delta_0$; in turn, applying \eqref{consist6}
to $u=\pa{x}$ with $\sigma=\{0\}$, we get
$\lambda_x=0$, which proves our claim.

Note that $X = \bigsqcup_{x\in X}
\phi^{-1}(\{x\})$ (the disjoint union). Hence,
the conditional expectation $\esf(f)$ of a
function $f \colon X \to \rbop$ with respect to
$\phi^{-1}(\ascr)$ is given by
   \begin{align} \label{leba}
\esf(f)(z) = \frac{\int_{\dzi{x}} f \D
\nu_w}{\hsf(x)}, \quad z \in \phi^{-1}(\{x\}), \,
x \in X_+,
   \end{align}
where $X_+ := \{x \in X\colon \nu_w(\dzi{x})
> 0\}$ (see also \eqref{leba2}); on the remaining part of $X$ we can put
$\esf(f) = 0$.

Substituting $\sigma=\{0\}$ into
\eqref{consist6}, we see that $\mu_y(\{0\})=0$
for every $y\in V^\circ$ such that $\lambda_y\neq
0$. Thus, using the standard measure-theoretic
argument and \eqref{consist6}, we deduce that
   \begin{align} \label{cc-dt}
\int_{\sigma} t \D \mu_x(t) = \sum_{y\in\dzi{x}}
|\lambda_y|^2 \mu_y(\sigma), \quad \sigma \in
\borel{\rbb_+}, \quad x\in X.
   \end{align}
Set $P(x, \sigma) = \mu_x(\sigma)$ for $x\in X$
and $\sigma \in \borel{\rbb_+}$. It follows from \eqref{leba} and \eqref{cc-dt} that $P\colon X
\times \borel{\rbb_+} \to [0,1]$ is a ($\ascr$-measurable) family of probability
measures which fulfils the following
equality
   \begin{align} \label{esfcc}
\esf(P(\cdot, \sigma)) (z) = \frac{\int_{\sigma}
t P(\phi(z),\D t)}{\hsf(\phi(z))}, \quad z \in \phi^{-1}(\{x\}), \, x \in X_+.
   \end{align}
This implies that $P$ satisfies \eqref{cc}.
Hence, by Theorem \ref{general}, the weighted
composition operator $\cfw$ (see \eqref{wco-def})
is subnormal. Since $\slam=\cfw$, assertion (i)
is proved.

(ii) It is easily seen that if $\mu$ is a finite
positive Borel measure on $\rbb_+$ and
$\int_0^\infty s^{n} \D \mu(s) < \infty$ for some
$n \in \nbb$, then $\int_0^\infty s^{k} \D \mu(s)
< \infty$ for every $k \in \nbb$ such that $k \Le
n$. This fact combined with Lemma \ref{ddn} and
\cite[Lemmata 2.3.1(i) and 4.2.2(i)]{b-j-j-sA}
implies assertion (ii).
   \end{proof}
   \begin{rem}
Assume that $\slam$ is a densely defined weighted shift on a
countably infinite directed tree $\tcal=(V,E)$ with
weights $\lambdab=\{\lambda_v\}_{v \in V^{\circ}}$. A careful inspection of the proof of Theorem
\ref{wsi} reveals that if $\{\mu_x\}_{x \in X}$
(with $X=V$) is a system of Borel probability
measures on $\rbb_+$ which satisfies
\eqref{consist6}, then $\hsf > 0$ a.e.\
$[\nu_w]$, the family $P$ defined by $P(x, \cdot)
= \mu_x$ for $x\in X$ satisfies \eqref{cc} and
$\mu_x=\delta_0$ for every $x\in X \setminus
X_+$. We claim that if $\hsf > 0$ a.e.\
$[\nu_w]$ and $P\colon X \times \borel{\rbb_+}
\to [0,1]$ is any family of probability measures
which satisfies \eqref{cc}, then the system
$\{\tilde \mu_x\}_{x \in X}$ of probability
measures defined by
   \begin{align*}
\tilde \mu_x =
   \begin{cases}
P(x,\cdot) & \text{ if } x\in X_+,
   \\
\delta_0 & \text{ otherwise,}
   \end{cases}
   \end{align*}
satisfies \eqref{consist6} with $\{\tilde
\mu_x\}_{x\in X}$ in place of $\{\mu_x\}_{x\in
X}$. Indeed, \eqref{cc} implies \eqref{esfcc}. Hence, by \eqref{leba},
equality in \eqref{cc-dt} holds for every $x \in
X_+$ with $\mu_z=P(z,\cdot)$ for $z\in X$.
This implies via the standard measure-theoretic argument that equality in
\eqref{consist6} holds for every $u \in X_+$.
Since $\hsf > 0$ a.e.\ $[\nu_w]$, we deduce that
equality in \eqref{consist6} holds for every
$u\in X_+$ with $\{\tilde \mu_x\}_{x\in X}$ in
place of $\{\mu_x\}_{x\in X}$. Clearly, this is also the case for $u \in
X\setminus X_+$. Thus, our claim is proved.
   \end{rem}
   \section{Trees with one branching vertex}
Theorem \ref{wsi} will be applied in the case of
weighted shifts on leafless directed trees with
one branching vertex. First, we recall the models
of such trees (see Figure 1 below). For
$\eta,\kappa \in \zbb_+ \sqcup \{\infty\}$ with
$\eta \Ge 2$, we define the directed tree
$\tcal_{\eta,\kappa} = (V_{\eta,\kappa},
E_{\eta,\kappa})$ as follows (the symbol
``\,$\sqcup$\,'' denotes disjoint union of sets)
   \allowdisplaybreaks
   \begin{align*}
V_{\eta,\kappa} & = \big\{-k\colon k\in
J_\kappa\big\} \sqcup \{0\} \sqcup
\big\{(i,j)\colon i\in J_\eta,\, j\in \nbb\big\},
   \\
E_{\eta,\kappa} & = E_\kappa \sqcup
\big\{(0,(i,1))\colon i \in J_\eta\big\} \sqcup
\big\{((i,j),(i,j+1))\colon i\in J_\eta,\, j\in
\nbb\big\},
   \\
E_\kappa & = \big\{(-k,-k+1) \colon k\in
J_\kappa\big\},
   \end{align*}
where $J_n = \{k \in \nbb\colon k\Le
n\}$ for $n \in \zbb_+ \sqcup \{\infty\}$. Clearly, $\tcal_{\eta,\kappa}$ is
leafless and $0$ is its only branching vertex.
From now on, we write $\lambda_{i,j}$ instead of
the more formal expression $\lambda_{(i,j)}$
whenever $(i,j) \in V_{\eta,\kappa}$.

\vspace{3ex}

   \begin{tikzpicture}[scale = .6, transform shape]
   \tikzstyle{every node} = [circle,fill=gray!30]
   \node (e1kappa)[font=\footnotesize] at (-2,1)
   {$\lambda_{-l}$}; \node
   (e-1)[font=\footnotesize] at (0,1) {$
   \lambda_{-1}$}; \node (e10) at (3.5,1)
   {$\lambda_0$}; \node (e11) at (7,3)
   {$\lambda_{1,1}$}; \node (e12) at (10.5,3)
   {$\lambda_{1,2}$}; \node (e13) at (14,3)
   {$\lambda_{1,3}$}; \node[fill = none] (e1n) at
   (16,3) {};

   \node (e21) at (7,1) {$\lambda_{2,1}$}; \node
   (e22) at (10.5,1) {$\lambda_{2,2}$}; \node
   (e23) at (14,1) {$\lambda_{2,3}$}; \node[fill
   = none] (e2n) at (16,1) {};

   \node[fill = none] (e 5 1) at (7,-3) {};
   \node[fill = none] (e 5 n) at (16,-3) {};

   \node (e31) at (7,-1) {$\lambda_{3,1}$}; \node
   (e32) at (10.5,-1) {$\lambda_{3,2}$}; \node
   (e33) at (14,-1) {$\lambda_{3,3}$}; \node
   (e3n)[fill = none] at (16,-1) {};

   \draw[->] (e10) --(e11) node[pos=0.5,above =
   0pt,fill=none] {}; \draw[->] (e10) --(e21)
   node[pos=0.5, above = 0pt,sloped, fill=none]
   {}; \draw[dotted] (e10) --(e 5 1)
   node[pos=0.5,below = -10pt,sloped,fill=none]
   {}; \draw[->] (e10) --(e31) node[pos=0.5,
   below=0pt,fill=none] {};

   \draw[->] (e11) --(e12) node[pos=0.5,above =
   0pt,fill=none] {}; \draw[->] (e21) --(e22)
   node[pos=0.5,above = 0pt,fill=none] {};

   \draw[->] (e31) --(e32)
   node[pos=0.5,below=0pt,fill=none] {};

   \draw[->] (e12) --(e13) node[pos=0.5,above =
   0pt,fill=none] {}; \draw[->] (e22) --(e23)
   node[pos=0.5,above=0pt,fill=none] {};
   \draw[->] (e32) --(e33) node[pos=0.5,below =
   0pt,fill=none] {};; \draw[dotted] (e 5 1) --(e
   5 n) node[pos=0.5,above = -10pt,fill=none] {};

   \draw[ ->] (e-1) --(e10) node[pos=0.5,above =
   0pt,fill=none] {};; \draw[dashed, ->] (e13)
   --(e1n); \draw[dashed, ->] (e23) --(e2n);
   \draw[dashed, ->] (e33) --(e3n); \draw[dashed,
   ->] (e1kappa) -- (e-1);
   \end{tikzpicture}
   \begin{center}
Figure 1
   \end{center}

\vspace{1ex}

   \begin{thm}\label{glowne}
Let $\eta,\kappa \in \zbb_+ \sqcup \{\infty\}$ be
such that $\eta \Ge 2$ and let $\slam$ be a
weighted shift on a directed tree
$\tcal_{\eta,\kappa}$ with nonzero weights
$\lambdab = \{\lambda_v\}_{v \in
V_{\eta,\kappa}^{\circ}}$. Suppose that there
exists a sequence $\{\mu_i\}_{i=1}^\eta$ of Borel
probability measures on $\rbb_+$ such that
   \begin{align} \label{zgod0}
\int_0^\infty s^n \D \mu_i(s) =
\Big|\prod_{j=2}^{n+1}\lambda_{i,j}\Big|^2, \quad
n \in \nbb, \; i \in J_\eta,
   \end{align}
and that one of the following three disjunctive
conditions is satisfied\/{\em :}
   \begin{enumerate}
   \item[(i)]  $\kappa=0$ and
   \begin{align*}
\sum_{i=1}^\eta |\lambda_{i,1}|^2 \int_0^\infty
\frac 1 s\, \D \mu_i(s) \Le 1,
   \end{align*}
   \item[(ii)] $0 < \kappa < \infty$ and
   \begin{align} \label{zgod'}
\sum_{i=1}^\eta |\lambda_{i,1}|^2 \int_0^\infty
\frac 1 s\, \D \mu_i(s) &= 1,
   \\
\Big|\prod_{j=0}^{l-1} \lambda_{-j}\Big|^2
\sum_{i=1}^\eta|\lambda_{i,1}|^2 \int_0^\infty
\frac 1 {s^{l+1}} \D \mu_i(s) & = 1, \quad l \in
J_{\kappa-1}, \label{widly1}
   \\
\Big|\prod_{j=0}^{\kappa-1}
\lambda_{-j}\Big|^2\sum_{i=1}^\eta|\lambda_{i,1}|^2
\int_0^\infty \frac 1 {s^{\kappa+1}} \D \mu_i(s)
& \Le 1, \label{widly1'}
   \end{align}
   \item[(iii)] $\kappa=\infty$ and equalities \eqref{zgod'}
and \eqref{widly1} are valid.
   \end{enumerate}
Then the following two assertions hold\/{\em :}
\begin{enumerate}
\item[(a)] if $\slam$ is
densely defined, then $\slam$ is subnormal,
\item[(b)] if $n\in \nbb$, then $\slam^n$ is densely defined
if and only if
   \begin{align} \label{nd}
\sum_{i=1}^\eta |\lambda_{i,1}|^2 \int_0^\infty
s^{n-1} \, \D \mu_i(s) < \infty.
   \end{align}
\end{enumerate}
   \end{thm}
   \begin{proof}
As in the proof of \cite[Theorem 4.1]{b-j-j-sB},
we define the system $\{\mu_v\}_{v\in
V_{\eta,\kappa}}$ of Borel probability measures
on $\rbb_+$ and verify that $\{\mu_v\}_{v\in
V_{\eta,\kappa}}$ satisfies
\eqref{consist6}. Hence, assertion (a) is a direct consequence of Theorem \ref{wsi}(i).

(b) Fix $n\in\nbb$. It follows from Theorem \ref{wsi}(ii) that $\slam^n$ is densely defined if and only if
$\int_0^\infty s^n \, \D \mu_0(s) < \infty$. Using the explicit  definition of $\mu_0$ and applying the standard measure-theoretic argument, we see that
\begin{align*}
\int_0^\infty s^n \, \D \mu_0(s) = \sum_{i=1}^\eta |\lambda_{i,1}|^2 \int_0^\infty
s^{n-1} \, \D \mu_i(s).
\end{align*}
This completes the proof of
assertion (b) (the case of $n=1$ can also be settled without using the definition of $\mu_0$ simply by applying Lemma \ref{ddn} and \cite[Proposition 3.1.3(iii)]{j-j-s}).
    \end{proof}
Note that Theorem \ref{glowne} remains true if
its condition (ii) is replaced by the condition
(iii) of \cite[Theorem 4.1]{b-j-j-sB} (see also
\cite[Lemma 4.2]{b-j-j-sB} and its proof).
   \begin{cor} \label{wndd}
Under the assumptions of Theorem {\em
\ref{glowne}}, if $n\in \nbb$, then
the following two assertions are equivalent{\em
:}
   \begin{enumerate}
   \item[(i)] $\slam^n$ is densely defined and
$\slam^{n+1}$ is not,
   \item[(ii)] the condition \eqref{nd} holds and
$\sum_{i=1}^\eta |\lambda_{i,1}|^2 \int_0^\infty
s^n \, \D \mu_i(s) = \infty$.
   \end{enumerate}
   \end{cor}
   \section{The example}
It follows from \cite[Lemma 2.3.1(i)]{b-j-j-sA}
that if $\slam$ is a weighted shift on
$\tcal_{\eta,\kappa}$ and $\eta < \infty$, then
$\dzn{\slam}$ is dense in
$\ell^2(V_{\eta,\kappa})$ (this means that
Corollary \ref{wndd} is interesting only if
$\eta=\infty$). If $\eta=\infty$, the situation
is completely different. Using Theorem
\ref{glowne} and Corollary \ref{wndd}, we show
that for every $n\in \nbb$ and for every $\kappa
\in \zbb_+ \sqcup \{\infty\}$, there exists a
subnormal weighted shift $\slam$ on
$\tcal_{\infty,\kappa}$ such that $\slam^n$ is
densely defined and $\slam^{n+1}$ is not. For
this purpose, we adapt \cite[Procedure
6.3.1]{j-j-s} to the present context. In the
original procedure, one starts with a sequence
$\{\mu_i\}_{i=1}^\infty$ of Borel probability
measures on $\rbb_+$ (whose $n$th moments are
finite for every $n \in \zbb$ such that $n \Ge
-(\kappa +1)$) and then constructs a system of
nonzero weights $\lambdab=\{\lambda_v\}_{v \in
V_{\infty,\kappa}^{\circ}}$ that satisfies the
assumptions of Theorem \ref{glowne} (in fact,
using Lemma \ref{slon4} below, we can also
maintain the condition \eqref{nd}). However, in
general, it is not possible to maintain the
condition (ii) of Corollary \ref{wndd} even if
$\{\mu_i\}_{i=1}^\infty$ are measures with
two-point supports (this question is not
discussed here).
   \begin{exa} \label{glownyprz}
Assume that $\eta = \infty$. Consider the
measures $\mu_i = \delta_{q_i}$ with $q_i \in
(0,\infty)$ for $i \in \nbb$. By \cite[Notation
6.1.9 and Procedure 6.3.1]{j-j-s}, $\slam \in
\ogr{\ell^2(V_{\infty,\kappa})}$ if and only if
$\sup\big\{q_i\colon i\in\nbb \big\} < \infty$.
Hence, there is no loss of generality in assuming
that $\sup\big\{q_i\colon i\in\nbb \big\} =
\infty$. To cover all possible choices of $\kappa
\in \zbb_+ \sqcup \{\infty\}$, we look for a
system of nonzero weights $\{\lambda_v\}_{v \in
V_{\infty,\infty}}$ which satisfies
\eqref{zgod0}, \eqref{zgod'}, \eqref{widly1} with
$\kappa=\infty$, \eqref{nd} and the equality
$\sum_{i=1}^{\infty} |\lambda_{i,1}|^2
\int_0^\infty s^n \, \D \mu_i(s) = \infty$.
Setting $\lambda_{i,1} = \sqrt{\alpha_i q_i}$ for
$i\in\nbb$, we reduce our problem to finding a
sequence $\{\alpha_i\}_{i=1}^\infty \subseteq
(0,\infty)$ such that \allowdisplaybreaks
   \begin{align}  \label{slon2}
& \sum_{i=1}^\infty \alpha_i q_i^{l} < \infty, \quad l
\in \zbb \text{ and } l \Le n,
   \\  \label{slon3}
& \sum_{i=1}^\infty \alpha_i q_i^{n+1} = \infty.
   \end{align}
Indeed, if $\{\alpha_i\}_{i=1}^\infty$ is such a
sequence, then multiplying its terms by an appropriate
positive constant, we may assume that
$\{\alpha_i\}_{i=1}^\infty$ satisfies \eqref{slon2},
\eqref{slon3} and \eqref{zgod'}. Next we define the
weights $\big\{\lambda_{-j}\colon j\in \zbb_+\big\}$ recursively so as to satisfy
\eqref{widly1} with
$\kappa=\infty$, and finally we set
$\lambda_{i,j} = \sqrt{q_i}$ for all $i, j \in \nbb$
such that $j \Ge 2$. The so constructed weights
$\{\lambda_v\}_{v \in V_{\infty,\infty}}$
meets our requirements.

The following lemma turns out to be helpful when
solving the reduced problem.
   \begin{lem} \label{slon4}
If\/ $[a_{i,j}]_{i,j=1}^\infty$ is an infinite matrix
with entries $a_{i,j}\in \rbb_+$, then there exists a
sequence $\{\alpha_i\}_{i=1}^\infty \subseteq
(0,\infty)$ such that
   \begin{align*}
\sum_{i=1}^\infty \alpha_i a_{i,j} < \infty, \quad j
\in \nbb.
   \end{align*}
   \end{lem}
   \begin{proof}
First observe that for every $i \in \nbb$, there
exists $\alpha_i \in (0,\infty)$ such that $\alpha_i
\sum_{k=1}^i a_{i,k} \Le 2^{-i}$. Hence,
$\sum_{i=j}^\infty \alpha_i a_{i,j} \Le 1$ for every
$j\in \nbb$.
   \end{proof}
Since $\sup\big\{q_i\colon i\in\nbb \big\}=\infty$,
there exists a subsequence $\{q_{i_k}\}_{k=1}^\infty$
of the sequence $\{q_i\}_{i=1}^\infty$ such that
$q_{i_k} \Ge k$ for every $k\in \nbb$. Set $\varOmega
= \{i_k \colon k \in \nbb\}$. By Lemma \ref{slon4},
there exists $\{\alpha_i\}_{i\in \nbb \setminus
\varOmega} \subseteq (0,\infty)$ such that
   \begin{align} \label{slon5}
& \sum_{i \in \nbb \setminus \varOmega} \alpha_i
q_i^{l} < \infty, \quad l \in \zbb \text{ and } l \Le
n.
   \end{align}
Define the system $\{\alpha_i\}_{i\in \varOmega}
\subseteq (0,\infty)$ by
   \begin{align*}
\alpha_{i_k} = \frac{1}{k^{2}q_{i_k}^n}, \quad k \in
\nbb.
   \end{align*}
Since $q_{i_k} \Ge k$ for all $k\in \nbb$, we get
   \begin{align} \label{slon6}
\sum_{i \in \varOmega} \alpha_i q_i^l = \sum_{k =
1}^\infty \alpha_{i_k} q_{i_k}^l = \sum_{k = 1}^\infty
\frac{1}{k^2 q_{i_k}^{n-l}} \Le \sum_{k = 1}^\infty
\frac{1}{k^2} < \infty, \quad l \in \zbb \text{ and }
l \Le n,
   \end{align}
and
   \begin{align}  \label{slon7}
\sum_{i \in \varOmega} \alpha_i q_i^{n+1} = \sum_{k =
1}^\infty \alpha_{i_k} q_{i_k}^{n+1} = \sum_{k =
1}^\infty \frac{q_{i_k}}{k^2} \Ge \sum_{k = 1}^\infty
\frac{1}{k} = \infty.
   \end{align}
Combining \eqref{slon5}, \eqref{slon6} and
\eqref{slon7}, we get \eqref{slon2} and
\eqref{slon3}, which solves the reduced problem
and consequently gives the required example.
   \end{exa}
   \begin{rem} \label{lebaa}
It is worth mentioning that if $\kappa=\infty$, then
any weighted shift $\slam$ on $\tcal_{\infty,\infty}$
with nonzero weights is unitarily equivalent to an
injective composition operator in an $L^2$ space over
a $\sigma$-finite measure space (cf.\ \cite[Lemma
4.3.1]{j-j-s2}). This fact combined with Example
\ref{glownyprz} shows that for every $n\in \nbb$,
there exists a subnormal composition operator $C$ in
an $L^2$ space over a $\sigma$-finite measure space
such that $C^n$ is densely defined and $C^{n+1}$ is
not.
   \end{rem}
   
\end{document}